\documentclass[11pt,reqno]{amsart}
\usepackage{amssymb,amsmath,latexsym,amsthm}
\usepackage{mathrsfs}
\usepackage{tikz}
\usepackage{graphicx}
\usepackage{mathtools}

\usepackage{url}		
\usetikzlibrary{angles,quotes}
\usepackage{multicol}

\usepackage{empheq}
\usepackage{caption}

\calclayout

\makeatletter
\newcommand\binomialCoefficient[2]{%
    \c@pgf@counta=#1
    \c@pgf@countb=#2
    \c@pgf@countc=\c@pgf@counta%
    \advance\c@pgf@countc by-\c@pgf@countb%
    \ifnum\c@pgf@countb>\c@pgf@countc%
        \c@pgf@countb=\c@pgf@countc%
    \fi%
    \c@pgf@countc=1
    \c@pgf@countd=0
    \pgfmathloop
        \ifnum\c@pgf@countd<\c@pgf@countb%
        \multiply\c@pgf@countc by\c@pgf@counta%
        \advance\c@pgf@counta by-1%
        \advance\c@pgf@countd by1%
        \divide\c@pgf@countc by\c@pgf@countd%
    \repeatpgfmathloop%
    \the\c@pgf@countc%
}
\newcommand{\skipitems}[1]{%
  \addtocounter{\@enumctr}{#1}%
}
\makeatother

\theoremstyle{plain}
\numberwithin{equation}{section}
\newtheorem{thm}{Theorem}[section]

\newtheorem{remark}{Remark}
\newtheorem*{theorem*}{Theorem}
\newtheorem*{proposition*}{Proposition}

\newtheorem{corollary}[thm]{Corollary}

\title{A Combinatorial Proof for Partitions of\\Pythagorean Triples Into Three Parts}
\author{Ivan V. Morozov}
\date{3 July 2026}

\begin{document}

\begin{abstract}
Any Pythagorean triple $\{a,b,c\}$ such that $a^{2}+b^{2}=c^{2}$ satisfies an elegant relation between its partitions into three parts, namely $p(a,3)+p(b,3)=p(c,3)$. While this property follows from elementary analytic methods, we give the first combinatorial proof of this relation.
\end{abstract}

\maketitle

\section{Introduction}

The number of integer partitions of a natural number $n$ into exactly three parts is well-understood to be expressed by
\begin{equation*}
p(n,3)=\left[\frac{n^{2}+3}{12}\right]\,,
\end{equation*}
where $[\cdots]$ denotes rounding to the nearest integer \cite{hons}. It is easily verifiable from this formula that $p(a,3)+p(b,3)=p(c,3)$, where $\{a,b,c\}$ is a Pythagorean triple. However, the lack of a combinatorial proof and a homologous communication of this problem from Dr. B. Hopkins \cite{miller} prompted the construction of such proof presented in this paper.

For example, for the Pythagorean triple $\{9,12,15\}$ we have $p(9,3)+p(12,3)=7+12=19=p(15,3)$, and for $\{15,36,39\}$ we have $p(15,3)+p(36,3)=19+108=127=p(39,3)$. In fact, we can easily see that for the family of Pythagorean triples $\{a,b,c\}$ where $a\equiv\pm 3\pmod{12}$ and $b\equiv 0\text{ or }6\pmod{12}$, we get $p(a,3)+p(b,3)=\frac{a^{2}+3}{12}+\frac{b^{2}}{12}=\frac{c^{2}+3}{12}=p(c,3)$, which is true since $c\equiv\pm 3\pmod{12}$.

We first derive a combinatorial formula for the necessary cases of the $p(n,3)$ formula by creating a bijection with lattice points bounded by a particular right triangle. Then, we use this construction to prove the desired identity for all necessary cases.

\section{Combinatorial Derivation of the $p(n,3)$ Formula}

Suppose $n\geq 3$, and let $n=n_{1}+n_{2}+n_{3}$ be a partition of $n$ into $3$ parts, with $n_{1}\geq n_{2}\geq n_{3}\geq 1$. We also define $x=n_{1}-n_{2}$, $y=n_{2}-n_{3}$, and $z=n_{3}-1$, which produces the set
\begin{equation}
P(n)=\{(x,y,z)\in\mathbb{Z}^{3}_{\geq 0}:x+2y+3z=n-3\}\,.
\end{equation}
These are lattice points that lie within the triangular region $u+2v+3w=n-3$ for $(u,v,w)\in\mathbb{R}^{3}_{\geq 0}$. Note that there is a natural projection of $P(n)$ onto the $yz$-plane, which we will define as
\begin{equation}
P_{x}(n)=\{(y,z)\in\mathbb{Z}^{2}_{\geq 0}:2y+3z\leq n-3\}\,.
\end{equation}
\begin{remark}
$\lvert P(n)\rvert=p(n,3)$.
\end{remark}
Our definitions of $x,y,z$ yield a uniquely determined partition given any $(x,y,z)\in P(n)$. So, the map $(n_{1},n_{2},n_{3})\mapsto(x,y,z)=(n_{1}-n_{2},n_{2}-n_{3},n_{3}-1)$ is bijective and has the inverse $(x,y,z)\mapsto(n_{1},n_{2},n_{3})=(x+y+z+1,y+z+1,z+1)$. Thus $\lvert P(n)\rvert=p(n,3)$.
\begin{remark}
$\lvert P(n)\rvert=\lvert P_{x}(n)\rvert$.
\end{remark}
The map $P(n)\to P_{x}(n)$ that maps $(x,y,z)\mapsto(y,z)$ is bijective. Fixing $(y,z)\in P_{x}(n)$, we see that there exists a unique value $x=n-3-2y-3z\geq 0$ that satisfies $(x,y,z)\in P(c)$. Existence gives us surjection, and uniqueness injection, hence bijection.

Observe that $P_{x}(n)$ is the set of lattice points inside the triangular region $2v+3w\leq n-3$ for $(v,w)\in\mathbb{R}^{2}_{\geq 0}$. Moreover, this is a right triangle with vertices at $(0,0)$, $\left(\frac{n-3}{2},0\right)$, and $\left(0,\frac{n}{3}-1\right)$, so its hypotenuse is a segment of the line
\begin{equation}\label{1}
y=-\frac{3}{2}z+\frac{n-3}{2}\,,
\end{equation}
or, alternatively,
\begin{equation}\label{2}
z=-\frac{2}{3}y+\frac{n}{3}-1\,.
\end{equation}
Using equation \eqref{1}, for each nonnegative integer $0\leq z\leq\left\lfloor\frac{n}{3}\right\rfloor-1$, there are $\left\lfloor-\frac{3}{2}z+\frac{n-3}{2}\right\rfloor+1$ lattice points in $P_{x}(n)$ with that $z$-coordinate. Summing across $z$-coordinates we get
\begin{equation}
\lvert P_{x}(n)\rvert=p(n,3)=\sum_{z=0}^{\left\lfloor\frac{n}{3}\right\rfloor-1}\left(\left\lfloor-\frac{3}{2}z+\frac{n-3}{2}\right\rfloor+1\right)=\sum_{z=1}^{\left\lfloor\frac{n}{3}\right\rfloor}\left(\left\lfloor-\frac{3}{2}z+\frac{n}{2}\right\rfloor+1\right)
\end{equation}
by re-indexing. This is a well-known formula for $p(n,3)$, which is derivable from the fact that $p(n,2)=\left\lfloor\frac{n}{2}\right\rfloor$ and the recursion formula $p(n,k)=p(n-k,k)+p(n,k-1)$ \cite{skiena}. Alternatively, using equation \eqref{2}, we get an equivalent formula
\begin{equation}
\lvert P_{x}(n)\rvert=p(n,3)=\sum_{y=0}^{\left\lfloor\frac{n-3}{2}\right\rfloor}\left(\left\lfloor-\frac{2}{3}y+\frac{n}{3}-1\right\rfloor+1\right)=\sum_{y=1}^{\left\lfloor\frac{n-1}{2}\right\rfloor}\left\lfloor-\frac{2}{3}y+\frac{n+2}{3}\right\rfloor\,.
\end{equation}

\begin{corollary}\label{cor}
The formula $p(n,3)$ is expressible by cases of residues modulo $2$ and $3$:
\[
p(n,3)=\begin{cases} 
\frac{n^{2}+3}{12} & n\equiv 1\pmod{2},n\equiv 0\pmod{3} \\
\frac{n^{2}}{12} & n\equiv 0\pmod{2},n\equiv 0\pmod{3} \\
\frac{n^{2}-1}{12} & n\equiv 1\pmod{2},n\equiv\pm 1\pmod{3} \\
\frac{n^{2}-4}{12} & n\equiv 0\pmod{2},n\equiv\pm 1\pmod{3}
\end{cases}\,.
\]
\end{corollary}
\begin{proof}
Suppose $n\equiv 1\pmod{2}$ and $n\equiv 0\pmod{3}$. Then
\begin{align*}
p(n,3)&=\sum_{z=1}^{\left\lfloor\frac{n}{3}\right\rfloor}\left(\left\lfloor-\frac{3}{2}z+\frac{n}{2}\right\rfloor+1\right)=\sum_{z=1}^{\frac{n}{3}}\left(\left\lfloor-\frac{3}{2}z+\frac{1}{2}\right\rfloor+\frac{n-1}{2}+1\right)\\
&=\sum_{z=1}^{\frac{n}{3}}\left(\left\lfloor\frac{1}{2}(1-3z)\right\rfloor+\frac{n+1}{2}\right)\\
&=\frac{n+1}{2}\cdot\frac{n}{3}+\sum_{k=1}^{\frac{n+3}{6}}\frac{1}{2}(1-3(2k-1))+\sum_{k=1}^{\frac{n-3}{6}}\frac{1}{2}(-3(2k))\\
&=\frac{n^{2}+n}{6}-\frac{n^{2}+4n+3}{24}-\frac{n^{2}-9}{24}=\frac{n^{2}+3}{12}\,.
\end{align*}
Hence, this case is proven. The remaining cases can be shown by the same exact means.
\end{proof}

\section{The Combinatorial Proof of $p(a,3)+p(b,3)=p(c,3)$}

It is evident from the definition that if $n>m$ are two natural numbers, then $P_{x}(m)\subset P_{x}(n)$. We will show that if $\{a,b,c\}$ is a Pythagorean triple such that $a^{2}+b^{2}=c^{2}$, then $\lvert P_{x}(c)\setminus P_{x}(a)\rvert=\lvert P_{x}(b)\rvert$, which combinatorially proves $p(c,3)-p(a,3)=p(b,3)$, our desired identity.

For the general overview of the proof structure, we observe that $P_{x}(a)$ and $P_{x}(c)$ are bounded by similar right triangles, so $P_{x}(c)\setminus P_{x}(a)$ are lattice points bounded by a trapezoidal region. We count these lattice points in two steps: first we count the lattice points bounded by a parallelogram region inside $P_{x}(c)\setminus P_{x}(a)$. Then we count the remaining lattice points bounded by the remaining triangular region as half of the lattice points bounded by some rectangular region, with slight modifications depending on congruences of $a$ and $c$ modulo $2$ and $3$.

We will prove this by cases, namely
\begin{enumerate}
    \item $c\equiv 1\pmod{2}$ and $c\equiv 0\pmod{3}$,
    \item $c\equiv 0\pmod{2}$ and $c\equiv 0\pmod{3}$,
    \item $c\equiv 1\pmod{2}$ and $c\equiv 1\pmod{3}$,
    \item $c\equiv 0\pmod{2}$ and $c\equiv 1\pmod{3}$,
    \item $c\equiv 1\pmod{2}$ and $c\equiv 2\pmod{3}$,
    \item $c\equiv 0\pmod{2}$ and $c\equiv 2\pmod{3}$.
\end{enumerate}

\subsection*{The case $c\equiv 1\pmod{2}$ and $c\equiv 0\pmod{3}$}

Without loss of generality, let $a$ be the length of the leg satisfying $a\equiv 0\pmod{2}$ and $a\equiv 0\pmod{3}$, or $a\equiv 0\pmod{6}$, which must exist. We look at the trapezoidal region $P_{x}(c)\setminus P_{x}(a)=ABCD$ bounded by the vertices $A=\left(0,\frac{a}{3}-1\right)$, $B=\left(0,\frac{c}{3}-1\right)$, $C=\left(\frac{c-3}{2},0\right)$, and $D=\left(\frac{a-3}{2},0\right)$. We want to show that $\lvert P_{x}(c)\setminus P_{x}(a)\rvert=p(b,3)$ for the other leg $b$ of the Pythagorean triple.

For every integer $y$-coordinate $0\leq y\leq\frac{a-3}{2}-\frac{1}{2}$, there are $\frac{c-a}{3}$ lattice points in $ABCD$, which accounts for $\frac{a-2}{2}\cdot\frac{c-a}{3}=\frac{ac-a^{2}-2c+2a}{6}$ lattice points. For the remaining integer $y$-coordinates $\frac{a-2}{2}\leq y\leq\frac{c-3}{2}$, we can verify that there are $\frac{1}{2}\left(\frac{c-a}{3}+1\right)\left(\frac{c-a+1}{2}\right)$ lattice points in $ABCD$ by observing that they form a triangular region that contains half as many lattice points as a rectangular region with vertices $(0,0)$, $\left(0,\frac{c-a-1}{2}\right)$, $\left(\frac{c-a}{3},0\right)$, and $\left(\frac{c-a}{3},\frac{c-a-1}{2}\right)$. Thus we obtain that
\begin{align}
\lvert ABCD\rvert&=\frac{ac-a^{2}-2c+2a}{6}+\frac{1}{2}\left(\frac{c-a}{3}+1\right)\left(\frac{c-a+1}{2}\right)\\
&=\frac{2ac-2a^{2}-4c+4a}{12}+\frac{c^{2}-2ac+a^{2}+4c-4a+3}{12}\\
&=\frac{c^{2}-a^{2}+3}{12}\\
&=\frac{b^{2}+3}{12}\,.
\end{align}

Since this case forces $b\equiv 1\pmod{2}$ and $b\equiv 0\pmod{3}$, by Corollary \ref{cor} we have $p(c,3)-p(a,3)=\lvert P_{x}(c)\setminus P_{x}(a)\rvert=\lvert ABCD\rvert=\lvert P_{x}(b)\rvert=p(b,3)$, which proves this case of the theorem.

\subsection*{The case $c\equiv 0\pmod{2}$ and $c\equiv 0\pmod{3}$}

Let $a$ be the length of either of the two legs, which is guaranteed to satisfy $a\equiv 0\pmod{2}$ and $a\equiv 0\pmod{3}$, or $a\equiv 0\pmod{6}$. Similarly to the previous case, for every integer $y$-coordinate $0\leq y\leq\frac{a-3}{2}-\frac{1}{2}$, there are $\frac{c-a}{3}$ lattice points in $ABCD$, which accounts for $\frac{a-2}{2}\cdot\frac{c-a}{3}=\frac{ac-a^{2}-2c+2a}{6}$ lattice points. For the remaining integer $y$-coordinates $\frac{a}{2}\leq y\leq\frac{c-4}{2}$, we can verify that there are $\frac{1}{2}\left(\frac{c-a}{3}\right)\left(\frac{c-a}{2}+2\right)$ lattice points in $ABCD$ by observing that they form a triangular region that contains half as many lattice points as a rectangular region with vertices $(0,0)$, $\left(0,\frac{c-a}{2}+1\right)$, $\left(\frac{c-a}{3}-1,0\right)$, and $\left(\frac{c-a}{3}-1,\frac{c-a}{2}+1\right)$. Thus we obtain that
\begin{align}
\lvert ABCD\rvert&=\frac{ac-a^{2}-2c+2a}{6}+\frac{1}{2}\left(\frac{c-a}{3}\right)\left(\frac{c-a}{2}+2\right)\\
&=\frac{2ac-2a^{2}-4c+4a}{12}+\frac{c^{2}-2ac+a^{2}+4c-4a}{12}\\
&=\frac{c^{2}-a^{2}}{12}\\
&=\frac{b^{2}}{12}\,.
\end{align}

Since this case forces $b\equiv 0\pmod{2}$ and $b\equiv 0\pmod{3}$, by Corollary \ref{cor} we have $p(c,3)-p(a,3)=\lvert P_{x}(c)\setminus P_{x}(a)\rvert=\lvert ABCD\rvert=\lvert P_{x}(b)\rvert=p(b,3)$, which proves this case of the theorem.

\subsection*{The case $c\equiv 1\pmod{2}$ and $c\equiv 1\pmod{3}$}

Without loss of generality, let $a$ be the length of the leg satisfying $a\equiv 0\pmod{3}$, which must exist.

\subsubsection*{The subcase $a\equiv 0\pmod{2}$}

For every integer $y$-coordinate $y\in\left\{2,5,8,\ldots,\frac{a-8}{2}\right\}$, equivalently $\frac{a-6}{6}$ many of them, there are $\frac{c-a+2}{3}$ lattice points in $ABCD$. For every other integer $y$-coordinate $0\leq y\leq\frac{a-3}{2}-\frac{1}{2}$ not in $\left\{2,5,8,\ldots,\frac{a-8}{2}\right\}$, equivalently $\frac{a-2}{2}-\frac{a-6}{6}=\frac{a}{3}$ many of them, there are $\frac{c-a-1}{3}$ lattice points in $ABCD$. This accounts for $\frac{a-6}{6}\cdot\frac{c-a+2}{3}+\frac{a}{3}\cdot\frac{c-a-1}{3}=\frac{ac-a^{2}-2c+2a-4}{6}$ lattice points. For the remaining integer $y$-coordinates $\frac{a-2}{2}\leq y\leq\frac{c-3}{2}$, we can verify that there are $\frac{1}{2}\left(\frac{c-a-1}{3}\right)\left(\frac{c-a+1}{2}+1\right)+\frac{c-a+5}{6}$ lattice points in $ABCD$ by observing that they form a triangular region that contains half as many lattice points as a rectangular region with vertices $(0,0)$, $\left(0,\frac{c-a+1}{2}\right)$, $\left(\frac{c-a-1}{3}-1,0\right)$, $\left(\frac{c-a-1}{3}-1,\frac{c-a+1}{2}\right)$, with $\frac{c-a+5}{6}$ points undercounted. The same kind
of argument will be used in the remaining cases, so we will give suggestive formulations instead
of drawn out arguments, which the reader can verify. Thus we obtain that
\begin{align}
\lvert ABCD\rvert&=\frac{ac-a^{2}-2c+2a-4}{6}+\frac{1}{2}\left(\frac{c-a-1}{3}\right)\left(\frac{c-a+1}{2}+1\right)+\frac{c-a+5}{6}\\
&=\frac{2ac-2a^{2}-4c+4a-8}{12}+\frac{c^{2}-2ac+a^{2}+4c-4a+7}{12}\\
&=\frac{c^{2}-a^{2}-1}{12}\\
&=\frac{b^{2}-1}{12}\,.
\end{align}

Since this case forces $b\equiv 1\pmod{2}$ and $b\equiv\pm 1\pmod{3}$, by Corollary \ref{cor} we have $p(c,3)-p(a,3)=\lvert P_{x}(c)\setminus P_{x}(a)\rvert=\lvert ABCD\rvert=\lvert P_{x}(b)\rvert=p(b,3)$, which proves this case of the theorem.

\subsubsection*{The subcase $a\equiv 1\pmod{2}$}

For every integer $y$-coordinate $y\in\left\{2,5,8,\ldots,\frac{a-5}{2}\right\}$, equivalently $\frac{a-3}{6}$ many of them, there are $\frac{c-a+2}{3}$ lattice points in $ABCD$. For every other integer $y$-coordinate $0\leq y\leq\frac{a-3}{2}$ not in $\left\{2,5,8,\ldots,\frac{a-5}{2}\right\}$, equivalently $\frac{a-1}{2}-\frac{a-3}{6}=\frac{a}{3}$ many of them, there are $\frac{c-a-1}{3}$ lattice points in $ABCD$. This accounts for $\frac{a-3}{6}\cdot\frac{c-a+2}{3}+\frac{a}{3}\cdot\frac{c-a-1}{3}=\frac{ac-a^{2}-c+a-2}{6}$ lattice points. For the remaining integer $y$-coordinates $\frac{a-1}{2}\leq y\leq\frac{c-3}{2}$, there are $\frac{1}{2}\left(\frac{c-a+2}{3}\right)\left(\frac{c-a}{2}\right)$ lattice points in $ABCD$. Thus we obtain that
\begin{align}
\lvert ABCD\rvert&=\frac{ac-a^{2}-c+a-2}{6}+\frac{1}{2}\left(\frac{c-a+2}{3}\right)\left(\frac{c-a}{2}\right)\\
&=\frac{2ac-2a^{2}-2c+2a-4}{12}+\frac{c^{2}-2ac+a^{2}+2c-2a}{12}\\
&=\frac{c^{2}-a^{2}-4}{12}\\
&=\frac{b^{2}-4}{12}\,.
\end{align}

Since this case forces $b\equiv 0\pmod{2}$ and $b\equiv\pm 1\pmod{3}$, by Corollary \ref{cor} we have $p(c,3)-p(a,3)=\lvert P_{x}(c)\setminus P_{x}(a)\rvert=\lvert ABCD\rvert=\lvert P_{x}(b)\rvert=p(b,3)$, which proves this case of the theorem.

\subsection*{The case $c\equiv 0\pmod{2}$ and $c\equiv 1\pmod{3}$}

Without loss of generality, let $a$ be the length of the leg satisfying $a\equiv 0\pmod{2}$ and $a\equiv 0\pmod{3}$, or $a\equiv 0\pmod{6}$, which must exist. For every integer $y$-coordinate $y\in\left\{2,5,8,\ldots,\frac{a-8}{2}\right\}$, equivalently $\frac{a-6}{6}$ many of them, there are $\frac{c-a+2}{3}$ lattice points in $ABCD$. For every other integer $y$-coordinate $0\leq y\leq\frac{a-3}{2}-\frac{1}{2}$ not in $\left\{2,5,8,\ldots,\frac{a-8}{2}\right\}$, equivalently $\frac{a-2}{2}-\frac{a-6}{6}=\frac{a}{3}$ many of them, there are $\frac{c-a-1}{3}$ lattice points in $ABCD$. This accounts for $\frac{a-6}{6}\cdot\frac{c-a+2}{3}+\frac{a}{3}\cdot\frac{c-a-1}{3}=\frac{ac-a^{2}-2c+2a-4}{6}$ lattice points. For the remaining integer $y$-coordinates $\frac{a-2}{2}\leq y\leq\frac{c-4}{2}$, there are $\frac{1}{2}\left(\frac{c-a+2}{3}\right)\left(\frac{c-a}{2}\right)+\frac{c-a+2}{6}$ lattice points in $ABCD$. Thus we obtain that
\begin{align}
\lvert ABCD\rvert&=\frac{ac-a^{2}-2c+2a-4}{6}+\frac{1}{2}\left(\frac{c-a+2}{3}\right)\left(\frac{c-a}{2}\right)+\frac{c-a+2}{6}\\
&=\frac{2ac-2a^{2}-4c+4a-8}{12}+\frac{c^{2}-2ac+a^{2}+4c-4a+4}{12}\\
&=\frac{c^{2}-a^{2}-4}{12}\\
&=\frac{b^{2}-4}{12}\,.
\end{align}

Since this case forces $b\equiv 0\pmod{2}$ and $b\equiv\pm 1\pmod{3}$, by Corollary \ref{cor} we have $p(c,3)-p(a,3)=\lvert P_{x}(c)\setminus P_{x}(a)\rvert=\lvert ABCD\rvert=\lvert P_{x}(b)\rvert=p(b,3)$, which proves this case of the theorem.

\subsection*{The case $c\equiv 1\pmod{2}$ and $c\equiv 2\pmod{3}$}

Without loss of generality, let $a$ be the length of the leg satisfying $a\equiv 0\pmod{3}$, which must exist.

\subsubsection*{The subcase $a\equiv 0\pmod{2}$}

For every integer $y$-coordinate $y\in\left\{0,3,6,\ldots,\frac{a-6}{2}\right\}$, equivalently $\frac{a}{6}$ many of them, there are $\frac{c-a-2}{3}$ lattice points in $ABCD$. For every other integer $y$-coordinate $0\leq y\leq\frac{a-3}{2}-\frac{1}{2}$ not in $\left\{0,3,6,\ldots,\frac{a-6}{2}\right\}$, equivalently $\frac{a-2}{2}-\frac{a}{6}=\frac{a-3}{3}$ many of them, there are $\frac{c-a+1}{3}$ lattice points in $ABCD$. This accounts for $\frac{a}{6}\cdot\frac{c-a-2}{3}+\frac{a-3}{3}\cdot\frac{c-a+1}{3}=\frac{ac-a^{2}-2c+2a-2}{6}$ lattice points. For the remaining integer $y$-coordinates $\frac{a-2}{2}\leq y\leq\frac{c-3}{2}$, there are $\frac{1}{2}\left(\frac{c-a+1}{3}\right)\left(\frac{c-a+1}{2}+1\right)$ lattice points in $ABCD$. Thus we obtain that
\begin{align}
\lvert ABCD\rvert&=\frac{ac-a^{2}-2c+2a-2}{6}+\frac{1}{2}\left(\frac{c-a+1}{3}\right)\left(\frac{c-a+1}{2}+1\right)\\
&=\frac{2ac-2a^{2}-4c+4a-4}{12}+\frac{c^{2}-2ac+a^{2}+4c-4a+3}{12}\\
&=\frac{c^{2}-a^{2}-1}{12}\\
&=\frac{b^{2}-1}{12}\,.
\end{align}

Since this case forces $b\equiv 1\pmod{2}$ and $b\equiv\pm 1\pmod{3}$, by Corollary \ref{cor} we have $p(c,3)-p(a,3)=\lvert P_{x}(c)\setminus P_{x}(a)\rvert=\lvert ABCD\rvert=\lvert P_{x}(b)\rvert=p(b,3)$, which proves this case of the theorem.

\subsubsection*{The subcase $a\equiv 1\pmod{2}$}

For every integer $y$-coordinate $y\in\left\{0,3,6,\ldots,\frac{a-3}{2}\right\}$, equivalently $\frac{a+3}{6}$ many of them, there are $\frac{c-a-2}{3}$ lattice points in $ABCD$. For every other integer $y$-coordinate $0\leq y\leq\frac{a-3}{2}$ not in $\left\{0,3,6,\ldots,\frac{a-3}{2}\right\}$, equivalently $\frac{a-1}{2}-\frac{a+3}{6}=\frac{a-3}{3}$ many of them, there are $\frac{c-a+1}{3}$ lattice points in $ABCD$. This accounts for $\frac{a+3}{6}\cdot\frac{c-a-2}{3}+\frac{a-3}{3}\cdot\frac{c-a+1}{3}=\frac{ac-a^{2}-c+a-4}{6}$ lattice points. For the remaining integer $y$-coordinates $\frac{a-1}{2}\leq y\leq\frac{c-3}{2}$, there are $\frac{1}{2}\left(\frac{c-a+1}{3}+1\right)\left(\frac{c-a}{2}\right)-\frac{c-a-2}{6}$ lattice points in $ABCD$. Thus we obtain that
\begin{align}
\lvert ABCD\rvert&=\frac{ac-a^{2}-c+a-4}{6}+\frac{1}{2}\left(\frac{c-a+1}{3}+1\right)\left(\frac{c-a}{2}\right)-\frac{c-a-2}{6}\\
&=\frac{2ac-2a^{2}-2c+2a-8}{12}+\frac{c^{2}-2ac+a^{2}+2c-2a+4}{12}\\
&=\frac{c^{2}-a^{2}-4}{12}\\
&=\frac{b^{2}-4}{12}\,.
\end{align}

Since this case forces $b\equiv 0\pmod{2}$ and $b\equiv\pm 1\pmod{3}$, by Corollary \ref{cor} we have $p(c,3)-p(a,3)=\lvert P_{x}(c)\setminus P_{x}(a)\rvert=\lvert ABCD\rvert=\lvert P_{x}(b)\rvert=p(b,3)$, which proves this case of the theorem.

\subsection*{The case $c\equiv 0\pmod{2}$ and $c\equiv 2\pmod{3}$}

Without loss of generality, let $a$ be the length of the leg satisfying $a\equiv 0\pmod{2}$ and $a\equiv 0\pmod{3}$, or $a\equiv 0\pmod{6}$, which must exist. For every integer $y$-coordinate $y\in\left\{0,3,6,\ldots,\frac{a-6}{2}\right\}$, equivalently $\frac{a}{6}$ many of them, there are $\frac{c-a-2}{3}$ lattice points in $ABCD$. For every other integer $y$-coordinate $0\leq y\leq\frac{a-3}{2}-\frac{1}{2}$ not in $\left\{0,3,6,\ldots,\frac{a-6}{2}\right\}$, equivalently $\frac{a-2}{2}-\frac{a}{6}=\frac{a-3}{3}$ many of them, there are $\frac{c-a+1}{3}$ lattice points in $ABCD$. This accounts for $\frac{a}{6}\cdot\frac{c-a-2}{3}+\frac{a-3}{3}\cdot\frac{c-a+1}{3}=\frac{ac-a^{2}-2c+2a-2}{6}$ lattice points. For the remaining integer $y$-coordinates $\frac{a-2}{2}\leq y\leq\frac{c-4}{2}$, there are $\frac{1}{2}\left(\frac{c-a+1}{3}+1\right)\left(\frac{c-a}{2}\right)$ lattice points in $ABCD$. Thus we obtain that
\begin{align}
\lvert ABCD\rvert&=\frac{ac-a^{2}-2c+2a-2}{6}+\frac{1}{2}\left(\frac{c-a+1}{3}+1\right)\left(\frac{c-a}{2}\right)\\
&=\frac{2ac-2a^{2}-4c+4a-4}{12}+\frac{c^{2}-2ac+a^{2}+4c-4a}{12}\\
&=\frac{c^{2}-a^{2}-4}{12}\\
&=\frac{b^{2}-4}{12}\,.
\end{align}

Since this case forces $b\equiv 0\pmod{2}$ and $b\equiv\pm 1\pmod{3}$, by Corollary \ref{cor} we have $p(c,3)-p(a,3)=\lvert P_{x}(c)\setminus P_{x}(a)\rvert=\lvert ABCD\rvert=\lvert P_{x}(b)\rvert=p(b,3)$, which proves this case of the theorem.

Having checked all possible cases, our combinatorial proof for the partition identity $p(a,3)+p(b,3)=p(c,3)$ is complete.

\section{Acknowledgment}

The author would like to thank Professor Satyanand Singh for his advisorship on the progress of this problem, as well as Professor Brian Hopkins for his communication of the problem during the Combinatorial and Additive Number Theory Conference 2025.

\end{document}